\documentclass[reqno]{amsart}
\usepackage{amsmath}
\usepackage{amssymb}
  \usepackage{paralist}
  \usepackage{graphics} 
 \usepackage[colorlinks=true]{hyperref}

\usepackage{mathrsfs}
\usepackage{extarrows}
\usepackage{enumerate}

\newtheorem{theorem}{Theorem}[section]
\newtheorem{corollary}[theorem]{Corollary}

\newtheorem{lemma}[theorem]{Lemma}
\newtheorem{proposition}[theorem]{Proposition}

\theoremstyle{definition}
\newtheorem{definition}[theorem]{Definition}
\newtheorem{remark}[theorem]{Remark}

\newcommand{\ep}{\varepsilon}

\newcommand{\RR}{\mathbb{R}}
\newcommand{\cc}{\mathcal{C}}
\newcommand{\NN}{\mathbb{N}}

\newcommand{\btt}{\mathbb{T}}
\newcommand{\ZZ}{\mathbb{Z}}
\newcommand{\ZZP}{\mathbb{Z}^+}

\newcommand{\cd}{\mathcal{D}}
\newcommand{\cC}{\mathcal{C}}
\newcommand{\cS}{\mathcal{S}}
\newcommand{\cs}{\mathcal{S}}
\newcommand{\rr}{\mathcal{R}}

\newcommand{\cg}{\mathcal{G}}
\newcommand{\cP}{\mathcal{P}}
\newcommand{\cp}{\mathcal{P}}

\newcommand{\cE}{\mathcal{E}}

\newcommand{\cU}{\mathcal{U}}

\newcommand{\cH}{\mathcal{H}}

\newcommand{\sk}{\mathscr{K}}

\newcommand{\cm}{\mathcal{M}}

\newcommand{\sss}{\mathscr{S}}
\newcommand{\sC}{\mathscr{C}}

\newcommand{\sH}{\mathscr{H}}
\newcommand{\cM}{\mathcal{M}}

\newcommand{\spp}{\mathscr{P}}

\newcommand{\pe}{P_{\exp}^\bot}

\newcommand{\sg}{\mathscr{G}}

\newcommand{\al}{\alpha}
\newcommand{\supp}{\mathrm{supp}}
\newcommand{\NE}{\mathrm{NE}}

\newcommand{\var}{\mathrm{Var}}

\title[Denseness of intermediate pressures]
      {Denseness of intermediate pressures for systems with the 
      Climenhaga-Thompson structures}

\author[Peng Sun]{}


 \email{sunpeng@cufe.edu.cn}

\begin{document}

 \maketitle\ 


\centerline{\scshape Peng Sun}
\medskip
{\footnotesize
 \centerline{China Economics and Management Academy}
   \centerline{Central University of Finance and Economics}
   \centerline{Beijing 100081, China}
} 

\bigskip

\begin{abstract}
We show that systems with the structure introduced by  Climenhaga and Thompson have
dense intermediate pressures and dense intermediate entropies of ergodic measures. The result applies to the Ma\~n\'e diffeomorphisms.
\end{abstract}


\section{Introduction}

Let $(X,d)$ be a compact metric space and $f:X\to X$ be a continuous map. 
Then $(X,f)$ is conventionally called a \emph{topological dynamical system} or just a
\emph{system}.
Denote by $\cm(X,f)$ the space of probability invariant measures of 
$(X,f)$ and by $\cm_e(X,f)$
the subset of the ergodic ones. 
It is believed that in most cases, positive topological entropy of $(X,f)$ 
implies a rich structure of $\cm(X,f)$. 
A conjecture of Katok on intermediate entropies
suggests us to study the set
$$\sH(X,f):=\{h_\mu(f):\mu\in\cm_e(X,f)\}$$
of metric entropies of ergodic measures.
Katok Showed that $\sH(X,f)\supset\left[0,h(f)\right]$ for $C^{1+\al}$ surface
diffeomorphisms (cf. \cite{Ka80}) and conjectured that this holds for smooth
systems in any
dimension. In the last decade, a number of partial results have been obtained.
See \cite{Sun09, Sun10, Ures, Sun12, QS, GSW, KKK, 
LO, Burg, YZ, Sunze, LSWW, Sunintent}.

In the recent works, we find specification-like properties
powerful tools to construct invariant sets and ergodic measures. It is shown
in \cite{Sunintent} that systems satisfying approximate product property
and asymptotic entropy expansiveness must have ergodic measures of arbitrary
intermediate entropies and arbitrary intermediate pressures. In fact, in
this case, for every continuous potential $\phi$, every
$\al$ in the interval $\left(\inf\{P_\mu(\phi): \mu\in\cm(X,f)\}, P(\phi)\right)$,
where $P_\mu(\phi)$ and $P(\phi)$ denote the pressure of $\mu$ and the topological
pressure of $(X,f)$ with respect to $\phi$,
the set
$$\cm_e^\al(\phi):=\{\mu\in\cm_e(X,f): 
P_\mu(\phi)=\al\}$$
is a dense $G_\delta$ subset in the compact metric subspace
$$\cm(\phi,\al):=\{\mu\in\cm(X,f): \int\phi d\mu\le\al\le P_\mu(\phi)\}.$$
This result indicates that the invariant measures of such systems 
have delicate structures and we may expect stronger results than Katok's
conjecture for a broad class of systems.

Recently, Climenhaga and Thompson developed a methodology to study uniqueness
of equilibrium states for certain partially hyperbolic systems \cite{CT12,
CT14, CT, CFT, CTmane}. The core of  their theory is a decomposition of orbit
segments into pieces satisfying gluing orbit properties and pieces of low
pressures. In this article, we illustrate that such a decomposition can also
produce ergodic measures of intermediate entropies and intermediate pressures.

\begin{theorem}\label{thmain}
Let $(X,f)$
be an asymptotically entropy expansive system and $\phi:X\to\RR$ be a continuous
potential.
Denote
$$\spp(\phi):=\left\{P_\mu(\phi):\mu\in\cm_e(X,f)\right\}$$
and
$$P^*(\phi):=\liminf_{n\to\infty}\sup
\left\{\frac1n\sum_{k=0}^{n-1}\phi\left(f^k(x)\right):x\in X\right\}.$$
Suppose that $(X,f,\phi)$ has the Climenhaga-Thompson structure.
 Then $\spp(\phi)$ is dense in the interval
$\left[P^*(\phi), P(\phi)\right]$.
In particular, $\sH(X,f)$ is dense in $[0,h(f)]$.
\end{theorem}

The Climenhaga-Thompson structure (see Definition \ref{def_ctst})
 requires an estimate
$\pe(\phi)<P(\phi)$ on the obstruction to expansivity.
We suspect that this estimate may not imply asymptotic
entropy expansiveness. If the assumption of asymptotic
entropy expansiveness is released, we have
$$\overline{\spp(\phi)}\supset
\left[\max\left\{P^*(\phi),\pe(\phi)\right\}, P(\phi)\right].$$
See Corollary \ref{co_noaee}.

As an application of Theorem \ref{thmain},
we consider the Ma\~n\'e family $\cU_{A,\rho,r}$, 
which is a class of DA (derived from Anosov)
maps introduced by Ma\~n\'e \cite{Ma78}. They are
$C^0$ perturbations of a hyperbolic toral automorphism $A:\btt^3\to\btt^3$,
robustly transitive, and partially hyperbolic  with 1-dimensional centers.
As described in \cite{CTmane}, for each $g\in\cU_{A,\rho,\lambda}$ we assume that:
\begin{enumerate}
\item $q$ is a fixed point of $A$ and its neighborhood $B(q,\rho)$
is the support of the perturbation.
\item $\lambda\in[0,1]$ such that if an orbit spends a proportion at least $\lambda$
of its time outside $B(q,\rho)$, then it contracts the vectors in the center direction.
\end{enumerate}
Let $\phi$ be an $\al$-H\"older potential on $\btt^3$ and 
$$\|\phi\|_\al:=\sup\left\{\frac{|\phi(x)-\phi(y)|}{d(x,y)^\al}:x,y\in X,
x\ne y\right\}$$
be its H\"older semi-norm. It is shown in \cite{CTmane} that there is a function
$L(\rho,\lambda)$ such that $L(\rho,\lambda)\to\infty$ as $\rho,\lambda\to 0$
and $(\btt^3,g,\phi)$ has the Climenhaga-Thompson structure as long as
$\|\phi\|_\al\le L(\rho,\lambda)$.
As $g$ is entropy expansive, 
Theorem \ref{thmain} leads to
the following corollary.

\begin{corollary}
Let $(\btt^3,g)$ be the a system in the Ma\~n\'e family $\cU_{A,\rho, \lambda}$
and $\phi$ be an $\al$-H\"older potential satisfying $\|\phi\|_\al\le L(\rho,\lambda)$.
Then 
$$\overline{\spp(\phi)}\supset\left[P^*(\phi), P(\phi)\right].$$
In particular, there are $\rho_0,\lambda_0>0$
such that for every $g\in\cU_{A,\rho_0,\lambda_0}$,
we have
$$\overline{\sH(\btt^3,g)}=[0,h(f)].$$
\end{corollary}

We believe that the space of invariant measures of a Ma\~n\'e diffeomorphism 
has a similar
structure as the one described in \cite{Sunintent}, hence $\spp(\phi)$ and
$\sH(\btt^3,g)$ include some non-degenerate intervals. However, it seems
that the denseness result we have obtained might be optimal under the current
framework. We are looking for new approaches.

Notions and results in this article naturally extends to the continuous-time case.

\section{Preliminaries}

Suppose that we are given a system $(X,f)$ and a continuous potential function
$\phi:X\to\RR$. In what follows we often omit $f$ in
the notations. Denote by 
$\ZZ^+$ the set of all positive integers and by $\NN$ the set of all nonnegative integers, i.e.
$\NN=\ZZ^+\cup\{0\}$. 
For $n\in\ZZ^+$,  denote
$$\ZZ_n:=\{0,1,\cdots, n-1\}$$
and let
$\Sigma_n:=\ZZ_n^{\ZZ^+}$ be the space of all sequences in $\ZZ_n$.

\subsection{Entropy and Pressure}
\begin{definition}
Let $K$ be a subset of $X$.
For $n\in\ZZ^+$ and $\ep>0$, a subset $E\subset K$ 
is called an \emph{$(n,\ep)$-separated set} in $K$ if
for any distinct points $x,y$ in $E$, 
we have
$$d_n(x,y):=\max\left\{d(f^k(x),f^k(y)): k\in\ZZ_n\right\}>\ep.$$
Denote by $\sss(K,n,\ep)$ the maximal cardinality of $(n,\ep)$-separated subsets
of $K$. 
Let
$$h(K,f,\ep):=\limsup_{n\to\infty}\frac{\ln \sss(K,n,\ep)}{n}.$$
Then the \emph{topological entropy} of $f$ on $K$ is defined as
$$h(K,f):=\lim_{\ep\to0}h(K,f,\ep)=\sup\{h(K,f,\ep):\ep>0\}.$$
In particular, $h(f):=h(X,f)$ is called the topological entropy of the system $(X,f)$.
\end{definition}


 For $x\in X$, 
$\ep>0$ and $n\in\ZZ^+$, denote
$$\Phi(x,n,\ep):=\sup\left\{\sum_{k=0}^{n-1}
\phi\left(f^k\left(y\right)\right):y\in B_n(x,\ep)\right\},$$
where
$$B_n(x,\ep):=\{y\in X: d_n(x,y)<\ep\}.$$
In particular, we write
$$\Phi(x,n)=\Phi(x,n,0):=\sum_{k=0}^{n-1}
\phi\left(f^k\left(x\right)\right).$$
For $Y\subset X$,
we consider the \emph{partition function}
$$\Theta(Y,\phi, n, \delta,\ep):=\sup\left\{
\sum_{x\in E}e^{\Phi(x,n,\ep)}: E\text{ is an $(n,\delta)$-separated subset
of }Y\right\}.$$
Given $\cc\in X\times\NN$, denote 
$\cc^{(n)}:=\{x\in X: (x,n)\in\cd\}$.
We write 
$$\Theta(\cc,\phi, n, \delta,\ep):=\Theta(\cc^{(n)},\phi, n, \delta,\ep)
\text{ and }
\Theta(\cd,\phi,n,\delta):=\Theta(\cd,\phi, n, \delta,0).$$
Denote
$$P(\cd,\phi,\delta,\ep):=\limsup_{n\to\infty}\frac1n\ln\Theta(\cd,\phi,n,\delta,\ep)
\text{ and }P(\cd,\phi,\delta):=P(\cd,\phi,\delta,0).$$
The pressure of $\phi$ on $\cd$ is defined as
$$P(\cd,\phi):=\lim_{\delta\to 0} P(\cd,\phi,\delta).$$
In particular, for $Y\subset X$, we denote
$$P(Y,\phi,\delta,\ep):=P(Y\times\NN,\phi,\delta,\ep)$$
and $P(Y,\phi,\delta)$, $P(Y,\phi)$ analogously.
For $\cd=X\times\NN$ we obtain the usual
topological pressure for the system $(X,f)$, which shall be
simply denoted by $P(\phi,\delta,\ep)$, $P(\phi,\delta)$ and $P(\phi)$.

\begin{remark}
It holds that
$$\left|P(\cC,\phi,\delta,\ep)-P(\cC,\phi,\delta)\right|\le\var(\phi,\ep),$$
where
$$\var(\phi,\ep):=\sup\left\{\left|\phi(x)-\phi(y)\right|:d(x,y)\le\ep, x,y\in
X\right\}.$$
By continuity of $\phi$ and compactness of $X$, we have
$\var(\phi,\ep)<\infty$ for any $\ep>0$ and
$$\lim_{\ep\to 0}\var(\phi,\ep)=0.$$
Therefore it holds that
\begin{equation*}
\lim_{\ep\to0}P(\cC,\phi,\delta,\ep)=P(\cC,\phi,\delta).
\end{equation*}
\end{remark}

\begin{definition}
For $\cc\subset X\times\ZZ^+$ and $\ep>0$, 
we say that the potential $\phi$ satisfies BP$(\ep)$
on $\cc$ if
$$V(\cc,\phi,\ep):=\sup\left\{\left|\Phi(x,n)-\Phi(y,n)\right|:(x,n)\in\cc,
y\in B_n(x,\ep)\right\}<\infty.$$

We say that $\phi$ satisfies the \emph{Bowen property} on $\cc$ if 
$\phi$ satisfies BP$(\ep_0)$ on $\cc$ for some $\ep_0>0$.

\end{definition}

\begin{remark}\label{re_pressurebowen}
If $\phi$ satisfies BP$(\ep_0)$ on $\cC$ then
$$P(\cC,\phi,\delta,\ep_0)=P(\cC,\phi,\delta).$$
\end{remark}

\begin{proposition}[Variational Principle]
It holds that
$$P(\phi)=\sup\{P_\mu(\phi):\mu
\in\cm_e(X,f)\},$$
where $P_\mu(\phi):=h_\mu(f)+\int\phi d\mu$ is the pressure of $\mu$ and
$h_\mu(f)$ is the metric entropy of $\mu$.

\end{proposition}

Readers are referred to \cite{Walters} for more details on entropies and
pressures.

\subsection{Expansiveness}

\begin{definition}
For $\ep>0$ and $x\in X$, denote
$$\Gamma_\ep(x):=\{y\in X:d(f^n(x),f^n(y))<\ep\text{ for every }n\in\NN\}.$$
Let
$$h^*(f,\ep):=\sup\{h(\Gamma_\ep(x),f):x\in X\}.$$
\begin{enumerate}

\item We say that $(X,f)$ is \emph{expansive} if there is $\ep_0>0$ such
that $\Gamma_{\ep_0}(x)=\{x\}$ for every $x\in X$.
\item We say that $(X,f)$ is \emph{entropy expansive} if there is $\ep_0>0$
such that $h^*(f,\ep_0)=0$.
\item 
We say that $(X,f)$ is \emph{asymptotically entropy expansive} if
$$\lim_{\ep\to 0}h^*(f,\ep)=0.$$
\end{enumerate}
\end{definition}

\begin{lemma}[{\cite[Lemma 2.2]{CFT}}]\label{le_eppres}
For any $\cd\subset X\times\NN$ and $\ep>0$, we have
$$P(\cd,\phi)\le P(\cd,\phi,\ep)+h^*(f,\ep)+\var(\phi,\ep).$$
\end{lemma}

\begin{lemma}\label{le_uppercts}
Suppose that $(X,f)$ is asymptotically entropy expansive, then
the map $\mu\to P_\mu(\phi)$ is upper semi-continuous. In particular,
there is an ergodic measure $\mu_m$, called an equilibrium state for $\phi$,
such that $P_{\mu_m}(\phi)=P(\phi)$.
\end{lemma}

\begin{definition}
For $\ep>0$, denote
$$\NE(\ep):=\{x\in X: \Gamma_\ep(x)\ne\{x\}\}.$$
We say that $\mu\in\cm_e(X,f)$ is \emph{almost expansive}
if there is $\ep_0>0$ such that $\mu(\NE(\ep_0))=0$.
\end{definition}

\begin{definition}
For $\ep>0$, denote
\begin{align*}
\pe(\phi,\ep):=&\sup\{P_\mu(\phi):\mu(\NE(\ep))>0,\mu\in\cm_e(X,f)\}\\
=&\sup\{P_\mu(\phi):\mu(\NE(\ep))=1,\mu\in\cm_e(X,f)\}.
\end{align*}
The \emph{pressure of obstructions to expansivity} is defined as
$$\pe(\phi):=\lim_{\ep\to 0}\pe(\phi,\ep)=\inf\{\pe(\phi,\ep):\ep>0\}.$$
\end{definition}


\begin{lemma}[{\cite[Proposition 3.7]{CT}}]\label{le_pest}
If $P_{\exp}^\bot(\phi,\ep)<P(\phi)$, 
then $P(\phi,\gamma)=P(\phi)$ for $\gamma\in(0,\frac\ep2)$.
\end{lemma}

\begin{remark}\label{re_pest}
Let $Y$ be a compact invariant subset of $X$.
As a corollary of Lemma \ref{le_pest}, if $\pe(\phi,\ep)<P(Y,\phi)$,
then we have
$P(Y,\phi,\gamma)=P(Y,\phi)$ for $\gamma\in(0,\frac\ep2)$.
Just consider the restriction of $f$ to $Y$.
\end{remark}

\subsection{The Climenhaga-Thompson structure}
\begin{definition}\label{gapshadow}
Let $\cc=\{(x_k,m_k)\}
_{k=1}^\infty$ be a sequence in $X\times\ZZP$ and
$\rr=\{r_k\}
_{k\in\ZZ^+}$ be a sequence in $\NN$.
For $\ep>0$ and $z\in X$, we say that $(\cc,\rr)$ is \emph{$\ep$-traced} by $z$
if 
for each $k\in\ZZ^+$,
\begin{equation*}
d(f^{t_{k}+j}(z), f^j(x_k))\le\ep\text{ for each }j=0,1,\cdots, m_k-1,
\end{equation*}
where
$$t_1=t_1(\cc,\rr):=0\text{ and }
t_k=t_k(\cc,\rr):=\sum_{i=1}^{k-1}(m_i+r_i)\text{ for }k\ge 2.$$
\end{definition}

\begin{definition}
For $\cc\subset X\times\ZZ^+$, we say that $(\cc,f)$ 
satisfies GO$(\delta,N,\tau)$
if for every sequence 
$\{(x_k,m_k)\}_{k=1}^\infty$ in $\cc$
satisfying $m_k\ge N$ for all $k$, 
there are a sequence $\rr\in\Sigma_\tau$ and $z\in X$
such that 
$(\cc,\rr)$ is $\delta$-traced by $z$.

We say that $(\cc,f)$ satisfies GO$(\delta)$ if there are $N_0\in\NN$ and $\tau>0$
such that $(\cc,f)$ satisfies GO$(\delta,N_0,\tau_0)$. We say that $(\cc,f)$
satisfies the \emph{gluing orbit property} if it satisfies GO$(\delta)$ for every $\delta>0$.
\end{definition}

\begin{remark}
The gluing orbit property is called tail (W)-specification in 
\cite{CT} by Climenhaga and Thompson. 
We follow the name suggested by Bomfim
and Varandas \cite{BV} which may avoid ambiguity 
with the regular specification
property. Readers are also referred to 
\cite{BTV0, BTV, CLT, DGS, KLO, Sun19, Sunintent, Sununierg} 
for more discussions on gluing
orbit property and other specification-like properties, as well as various
examples.
\end{remark}

\begin{definition}\label{def_ctdecomp}
Let $\cd, \cP,\cg,\cS\subset X\times\NN$. We say that $(\cp,\cg,\cs)$ is a
\emph{CT-decomposition} for $\cd$ 
if it is associated with three functions $p,g,s:\cd\to\NN$ such that
for every $(x,n)\in\cd$, 
we have

$$(x,p(x,n))\in\cP,\;
 (f^{p(x,n)}(x),g(x,n))\in\cg,\;
 (f^{p(x,n)+g(x,n)}(x),s(x,n))\in\cS,$$
and
$$p(x,n)+g(x,n)+s(x,n)=n.$$

\end{definition}

\begin{lemma}[{\cite[Lemma 2.10]{CT}}]\label{le_gm}
Let $(\cp,\cg,\cs)$ be a CT-decomposition for $\cd$
such that $\cg$ satisfies the gluing orbit property.
For $M\in\NN$, denote
$$\cg_M:=\{(x,n)\in\cd: p(x,n)\le M\text{ and } s(x,n)\le M\}.$$
Then $\cg_M$ also satisfies the gluing orbit property.
\end{lemma}

\begin{definition}\label{def_ctst}
We say the $(X,f,\phi)$ has the \emph{Climenhaga-Thompson structure}
if there is $\cd\in X\times\NN$ and a CT-decomposition $(\cp,\cg,\cs)$ for
$\cd$
such that the following hold:
\begin{enumerate}
\item $(\cg,f)$ satisfies the gluing orbit property;
\item $P(\cd^c\cup\cP\cup\cS,\phi)<P(\phi)$;
\item $\phi$ has the Bowen property on $\cg$;
\item $P_{\exp}^\bot(\phi)<P(\phi)$.
\end{enumerate}
\end{definition}

\begin{definition}
Let $\delta,\gamma,\ep>0$.
We say the $(X,f,\phi)$ satisfies CT$(\delta,\gamma,\ep)$
if $$16\delta<8\gamma<\ep$$
and
 there is $\cd\in X\times\NN$ and a CT-decomposition $(\cp,\cg,\cs)$ for
$\cd$
such that the following hold:

\begin{enumerate}
\item $(\cg_M,f)$ satisfies GO$(\delta)$ for every $M\in\NN$;
\item $P(\cd^c,\phi,2\gamma,2\gamma)<P(\phi)$;
\item $P(\cP\cup\cS,\phi,\gamma,3\gamma)<P(\phi)$;

\item  $\phi$ satisfies BP$(3\gamma)$ on $\cg$;
\item $P_{\exp}^\bot(\phi,\ep)<P(\phi)$.
\end{enumerate} 

\end{definition}

\begin{remark}\label{re_ctimct}
By  Remark \ref{re_pressurebowen} and Lemma \ref{le_gm}, 
if $(X,f,\phi)$ has the Climenhaga-Thompson structure,
then there is $\ep_0>0$ such that  
$(X,f,\phi)$ satisfies CT$(\delta,\gamma,\ep)$ whenever
$$16\delta<8\gamma<\ep<\ep_0.$$
\end{remark}

\begin{lemma}[{\cite[Proposition 6.9]{CT}}]
\label{pr_gmest}
Suppose that $(X,f,\phi)$ satisfies CT$(\delta,\gamma,\ep)$.
Then there are $M$, $N_1$ and $C_0>0$ such that
for $n\ge N_1$, we have
$$\Theta(\cg_M, 2\gamma, n)\ge C_0e^{nP(\phi)}.$$
\end{lemma}




\section{Invariant Sets with Pressure Estimates}

In this section we prove the following proposition, which is the core
ingredient
of the proof of the main results. 


\begin{proposition}\label{pr_lambda}
Suppose that $(X,f,\phi)$ satisfies CT$(\delta,\gamma,\ep)$.
Then for every $\al\in(P^*(\phi),P(\phi))$ and every $\eta_0>0$,
there is a compact invariant set $\Lambda$ such that
$$|P(\Lambda,\phi,2\delta)-\al|<\eta_0.$$
\end{proposition}

\begin{proof}
See Lemma \ref{le_lowpresest} and Lemma \ref{le_upppresest}.
\end{proof}

\subsection{Construction}
We may assume that 
$$P^*(\phi)<\al-\eta_0<\al+\eta_0<P(\phi).$$
Let 
$$\eta:=\min\{\frac{\eta_0}{5},\frac{\al}{5}\}>0.$$ 
Denote
$$\phi^-:=\min\{\phi(x):x\in X\},$$
$$\phi^+:=\max\{\phi(x):x\in X\},$$ 
and
$$\var(\phi):=\phi^+-\phi^-.$$
We may assume that 
$$\phi^+\ge\phi^-\ge 0.$$ 
Otherwise, we can replace $\phi$
by $\phi-\phi^-$ then all conditions and results are parallel, 
as $P(\phi+c)=P(\phi+c)$ for any constant $c\in\RR$.
Let $M, N_1, C_0$ be as in Lemma \ref{pr_gmest}.
Suppose that $\cg_M$ satisfies GO$(\delta, N_0,\tau)$.

As $P^*(\phi)<\al-\eta$, 
we are able to find  and fix $N\in\NN$ such that
\begin{equation}\label{eq_phiest}
\sup\{\Phi(x,N):x\in X\}<N(\al-\eta)
\end{equation}
and $N$ is large enough such that the following holds:
\begin{align}\label{eq_nnest1}
&N>\max\{N_0, N_1\},\\
\label{eq_nnest2} 
&C_0e^{N\eta}>1,\;
e^{2N\eta}>2,\\
\label{eq_nnest3}
&N>\frac{\al\tau}{\eta}>\tau,\\
\label{eq_nnest4}
&N\eta>V(\cg,\phi,\delta)+2M\var(\phi),\\
\label{eq_nnest5}
&N\eta>\max\{\tau\phi^+,\ln\tau+\ln\sss(X,\tau,\delta)\}\ge\tau\phi^-.
\end{align}
Note that \eqref{eq_nnest3} guarantees that
\begin{equation}\label{eq_nnest3eq}
N(\al-\eta)>(N+\tau)(\al-2\eta)
\end{equation}

By Lemma \ref{pr_gmest} and \eqref{eq_nnest1}, we have
$$\Theta(\cg_M, 2\gamma, N)\ge C_0e^{NP(\phi)}.$$
There is an $(N,2\gamma)$-separated set $E^*\in\cg_M^{(N)}$ such that
$$\sum_{x\in E^*}e^{\Phi(x,N)}>C_0e^{N(\al+\eta_0)}>e^{N(\al+\eta)}.$$
By \eqref{eq_phiest} and \eqref{eq_nnest2}, 
we can find a subset $E$ of $E^*$ such that
\begin{equation}\label{eq_epresest}
e^{N(\al-\eta)}<\sum_{x\in E}e^{\Phi(x,N)}<e^{N(\al+\eta)}.
\end{equation}

Let $\sC=\{x_k(\sC)\}_{k=1}^\infty$ be a sequence in $E$.
Denote
$\tilde\sC:=\{(x_k(\sC),N)\}_{k=1}^\infty$.
Then $\tilde\sC$ is a sequence in $\cg_M$. As $\cg_M$ satisfies GO$(\delta, N_0,\tau_0)$
and \eqref{eq_nnest1} holds, there are a sequence
$\rr
\in\Sigma_{\tau}$ and $y\in X$ such that
$(\tilde\sC,\rr)$ is $\delta$-traced by $y$.
For $\rr=\{r_k(\rr)\}_{k=1}^\infty\in\Sigma_{\tau}$, we denote
$$t_1(\rr):=0\text{ and }
t_k(\rr):=\sum_{i=1}^{k-1}(N+r_i(\rr))\text{ for }k\ge 2.$$

Denote
$$Y(\sC,\rr):=\{y\in X:\text{$(\tilde\sC,\rr)$ is $\delta$-traced by $y$}\}.$$
It is possible that $Y(\sC,\rr)=\emptyset$. But for each $\sC\in E^{\ZZP}$
there is $\rr\in\Sigma_{\tau_0}$ such that $Y(\sC,\rr)\ne\emptyset$.
Let
$$Y:=\bigcup_{\sC\in E^{\ZZP}, \rr\in\Sigma_{\tau}}Y(\sC,\rr).$$

\begin{lemma}
$Y$ is compact.

\end{lemma}

\begin{proof}
Let $\{z_i\}_{i=1}^\infty$ be a sequence in $Y$ such that $z_i\to z$ in $X$.
We need to show that $z\in Y$.

Assume that $z_i\in Y(\sC_i,\rr_i)$
for each $i$.
Note that both $E^{\ZZP}$ and $\Sigma_\tau$ are compact metric 
symbolic spaces.
We can find a subsequence $\{n_j\}_{j=1}^\infty$, 
$\sC\in E^{\ZZP}$ and $\rr\in\Sigma_\tau$
such that 
$$\sC_{n_j}\to\sC\text{ and }\rr_{n_j}\to\rr.$$
For each $k\in\ZZ^+$, 
there is $M_k$ such that for every $n_j>M_k$,
we have
$$x_k(\sC_{n_j})=x_k(\sC)\text{ and }r_k(\rr_{n_j})=r_k(\rr).$$
This implies that 
\begin{align*}
d\left(f^{t_k(\rr)+l}\left(z\right), f^l\left(x_k\left(\sC\right)\right)\right)
=\lim_{n_j\to\infty} d\left(f^{t_k(\rr_{n_j})+l}\left(z_{n_j}\right), 
f^l\left(x_k\left(\sC_{n_j}\right)\right)\right)
\le\ep
\end{align*}
for each $l=0,\cdots, N-1$.
Hence $z\in Y(\sC,\rr)\subset Y$.
\end{proof}

Denote by $\sigma_1$ the shift map on $E^{\ZZP}$ and by $\sigma_2$ the shift
map on $\Sigma_\tau$.
For every $\sC\in E^{\ZZP}$ and $\rr\in\Sigma_\tau$, we have
\begin{equation}\label{eq_shiftinv}
f^{t_2(\rr)}(Y(\sC,\rr))\subset Y(\sigma_1(\sC),\sigma_2(\rr)).
\end{equation}

\begin{lemma}\label{cptinv}
Let
$$\Lambda:=\bigcup_{k=0}^{N+\tau-1} f^k(Y).$$
Then $\Lambda$ is a compact $f$-invariant subset. 

\end{lemma}

\begin{proof}

We have that $\Lambda$ is compact since $Y$ is compact.

For every $z\in\Lambda$, there is $y\in Y$ and $l\in\{0,\cdots,N+\tau-1\}$ such that
$f^l(y)=z$. 
If $l<N+\tau-1$, then 
$$f(z)=f^{\tau+1}(y)\in f^{\tau+1}(Y)\subset\Lambda.$$

Assume that $l=N+\tau-1$ and $y\in Y(\sC,\rr)$. 
By \eqref{eq_shiftinv}, we have
$f^{t_2(\rr)}(y)\in Y$ and $N\le t_2(\rr)\le N+\tau$.
Hence 
$$f(z)=f^{l+1}(y)=f^{N+\tau-t_2(\rr)}(f^{t_2(\rr)}(y))\in
f^{N+\tau-t_2(\rr)}(Y)\subset\Lambda.$$
This implies that $f(\Lambda)\subset\Lambda$.
\end{proof}

\subsection{Lower Estimate of Pressure}

\begin{lemma}\label{ngammasep}
Suppose that $y\in Y(\sC,\rr)$ and $y'\in Y(\sC',\rr')$
such that
$$t_n(\rr)=t_n(\rr')
\text{ and }
x_n(\sC)\ne x_n(\sC').$$
Then $y,y'$ are 
$(n(N+\tau),\gamma)$-separated. 
\end{lemma}

\begin{proof}
The tracing property implies that
$$d_N(f^{t_n(\rr)}(y), x_n(\sC))\le\delta
\text{ and }d_N(f^{t_n(\rr')}(y'), x_n(\sC'))\le\delta.$$
As $x_n(\sC), x_n(\sC')\in E\subset E^*$ and $x_n(\sC)\ne x_n(\sC')$, 
they are $(N,2\gamma)$-separated.
Hence we have
\begin{align*}
d_N(f^{t_n(\rr)}(y), f^{t_n(\rr')}(y'))>2\gamma-2\delta>\gamma.
\end{align*}
As $t_n(\rr)=t_n(\rr')$ and $t_n(\rr)+N\le (n-1)(N+\tau)+N<n(N+\tau)$,
this implies that $y,y'$ are 
$(n(N+\tau),\gamma)$-separated. 
\end{proof}

\begin{lemma}\label{le_lowpresest}
$P(Y,\phi,\gamma)\ge\al-\eta_0$.
\end{lemma}

\begin{proof}
For each finite sequence $(y_1,\cdots,y_n)\in E^n$, denote
\begin{equation}\label{eq_c1}
C_{y_1\cdots y_n}^1:=\{\sC\in E^{\ZZP}: x_k(\sC)=y_k\text{ for each } j=1,\cdots,n\}.
\end{equation}
For each finite sequence $(w_1,\cdots,w_n)\in(\ZZ_\tau)^n$, denote
\begin{equation}\label{eq_c2}
C_{w_1\cdots w_n}^2:=\{\rr\in\Sigma_\tau: r_k(\rr)=w_k\text{ for each } j=1,\cdots,n\}.
\end{equation}
Denote
$$Y(C_{y_1\cdots y_n}^1, C_{w_1\cdots w_{n-1}}^2):=
\bigcup_{\sC\in C_{y_1\cdots y_n}^1, \rr\in C_{w_1\cdots w_{n-1}}^2}Y(\sC,\rr)$$
and
$$E(w_1,\cdots, w_{n-1}):=\left\{(y_1,\cdots, y_n)\in E^n: 
Y(C_{y_1\cdots y_n}^1, C_{w_1\cdots w_{n-1}}^2)\ne\emptyset\right\}.$$

For each $n\in\ZZP$, 
there are $\tau^{n-1}$ elements in $\ZZ_\tau^{n-1}$.
By \eqref{eq_epresest} and the pigeonhole principle, there must be an element $(v_1,\cdots,v_{n-1})\in(\ZZ_\tau)^{n-1}$ such that
\begin{align}
\sum_{(y_1,\cdots, y_n)\in E(v_1,\cdots, v_{n-1})}(\prod_{k=1}^n e^{\Phi(y_k,N)})
\ge&\frac{\sum\limits_{(y_1,\cdots, y_n)\in E^n}(\prod_{k=1}^n e^{\Phi(y_k,N)})}
{\tau^{n-1}}\notag\\
=&\frac{\left(\sum_{x\in E}e^{\Phi(x,N)}\right)^n}{\tau^{n-1}}\notag\\
\label{eq_ecylpresest}\ge&\frac{e^{nN(\al-\eta)}}{\tau^{n-1}}.
\end{align}

For each $(y_1,\cdots, y_n)\in E(v_1,\cdots, v_{n-1})$, we can find a point
$$z=z(y_1,\cdots, y_n)\in Y(C_{y_1\cdots y_n}^1, C_{v_1\cdots v_{n-1}}^2).$$
As $\phi$ satisfies BP$(3\gamma)$ on $\cg$ and 
$y_k\in\cg_M$, by \eqref{eq_nnest4},
we have for each $k=1,\cdots, n$,
\begin{align}
&\left|\Phi\left(f^{t_k}(z),N\right)-\Phi\left(y_k,N\right)\right| \notag
\\\le&\left|\Phi(f^{t_k}(z),p(y_k,N))-\Phi(y_k,p(y_k,N))\right| \notag
\\&+\left|\Phi(f^{t_k+p(y_k,N)}(z),g(y_k,N))-\Phi(f^{p(y_k,N)}(y_k),g(y_k,N))\right|
\notag
\\&+\left|\Phi(f^{t_k+g(y_k,N)}(z),s(y_k,N))-\Phi(f^{g(y_k,N)}(y_k),s(y_k,N))\right|
\notag
\\ \le&V(\cg,\phi,\delta)+2M\var(\phi) \notag
\\\label{eq_gonvar}\le& N\eta,
\end{align}
where
$t_k:=\sum_{i=1}^{k-1}(N+v_i)$
and
$p,g,s$ are functions associated with the CT-decomposition $(\cp,\cg,\cs)$ as in Definition
\ref{def_ctdecomp}.
Hence, we have
\begin{align}
&\Phi\left(z,n\left(N+\tau\right)\right)\notag\\
=&
\sum_{k=1}^n\Phi\left(f^{t_k}(z),N\right)
+\sum_{k=1}^{n-1}\Phi\left(f^{t_k+N}\left(z\right),v_k\right)
+\Phi\left(f^{t_n+N}\left(z\right),\left(n\tau-\sum_{k=1}^{n-1}v_k\right)\right)
\notag\\\ge&\sum_{k=1}^n\left(\Phi\left(y_k,N\right)
-N\eta\right)
+n\tau\phi^- \notag
\\\label{eq_phizest}
\ge&\sum_{k=1}^n\Phi\left(y_k,N\right)+
n\left(\tau\phi^--N\eta\right)
\end{align}

Lemma \ref{ngammasep} implies that
$$E':=\{z(y_1, \cdots, y_n):(y_1,\cdots, y_n)\in E(v_1,\cdots, v_{n-1})\}$$
is an $(n(N+\tau),\gamma)$-separated set in $Y$. 
By \eqref{eq_phizest} and \eqref{eq_ecylpresest}, we have
\begin{align*}
&\Theta(Y,\phi,n(N+\tau),\gamma)
\\\ge&\sum_{z\in E'}e^{\Phi(z,n(N+\tau))}
\\\ge&\sum_{(y_1,\cdots, y_n)\in E(v_1,\cdots, v_{n-1})}
e^{\sum_{k=1}^n\Phi\left(y_k,N\right)+
n\left(\tau\phi^--N\eta\right)}
\\\ge&\frac{e^{nN(\al-\eta)}}{\tau^{n-1}}
e^{n\left(\tau\phi^--N\eta\right)}
\end{align*}

By \eqref{eq_nnest4}, \eqref{eq_nnest5} and \eqref{eq_nnest3eq}, we have
\begin{align*}
P(Y,\phi,\gamma)\ge&\limsup_{n\to\infty}
\frac{\ln\Theta(Y,\phi,n(N+\tau),\gamma)}{n(N+\tau)}
\\\ge&\frac{1}{N+\tau}\left(
N(\al-\eta)+\tau\phi^--N\eta-\ln\tau
\right)
\\\ge&(\al-2\eta)-\eta-\eta
\\>&\al-\eta_0.
\end{align*}

\end{proof}

\subsection{Upper Estimate of Pressure}

\begin{lemma}\label{uppnepest}
For   every $n\in\ZZ^+$, 
every $C_{y_1\cdots y_n}^1$ as in \eqref{eq_c1}
and  every 
$C^2_{w_1\cdots w_{n-1}}$ as in \eqref{eq_c2}
, there are at most
$\sss(X,\tau,\delta)^{n-1}$ points in 
$Y(C_{y_1\cdots y_n}^1, C_{w_1\cdots w_{n-1}}^2)$ that are $(nN,2\delta)$-separated.
\end{lemma}

\begin{proof}
Let $S$ be a fixed $(\tau,\delta)$-separated subset of $X$ with the maximal
cardinality. 
Then $|S|=\sss(X,\tau,\delta)$ and
\begin{equation}\label{eq_taudelspan}
X\subset\bigcup_{x\in S} B_\tau(x,\delta).
\end{equation}

Denote 
$$t_k:=\sum_{i=1}^{k-1}(N+w_i).$$
We define a set
$$\Omega:=\prod_{i=0}^{t_n+N-1}\Omega_i$$
such that
\begin{align*}
\Omega_i:=\begin{cases}
\{f^{i-t_k}(y_k)\}, &\text{ if } 0\le i-t_k<N\text{ for some }k;\\
f^{i-t_k-N}(S), & \text{ if } t_k+N\le i<t_{k+1}\text{ for some }k.
\end{cases}
\end{align*}
Note that in the latter case we have 
$$i-t_k-N\in\ZZ_{\tau}$$ as
$$t_{k+1}=t_k+N+w_k\le t_k+N+\tau.$$

Let $Q$ be an $(t_n+N,2\delta)$-separated set in 
$Y(C_{y_1\cdots y_n}^1, C_{w_1\cdots w_{n-1}}^2)$.
Then for every $y\in Q$, for each $k=1,\cdots, n-1$, by \eqref{eq_taudelspan}, there is
$z_k\in S$ such that $f^{t_k}(y)\in B_\tau(z_k,\delta)$.
Denote 
$$\pi(y)=(\pi(y)_1,\cdots,\pi(y)_{t_{n+1}(\xi)-1})
\text{ such that }\pi(y)_i\in\Omega_{i}\text{ for each }i$$
and for each $k=1,\cdots, n-1$ and $l\in\ZZ_{w_k}$ define
$$\pi(y)_{t_k+N+l}:=f^l(z_k).$$
Then $\pi:Q\to\Omega$ defines a function such that
$$d(f^i(y),\pi(y)_i)\le\delta\text{ for each }i=0,1\cdots, t_n+N-1.$$
As $Q$ is $(t_n+N,2\delta)$-separated, $\pi$ is an injection.
Hence, we have
$$|Q|\le|\Omega|\le|S|^{n-1}=\sss(X,\tau,\delta)^{n-1}.$$
The result follows as $t_n+N\ge nN$.
\end{proof}

\begin{lemma}\label{le_cylpresest}
Let $n>10$. Denote
\begin{equation}\label{eq_thetan}
\theta_n:=\lfloor\frac{(n-4)N}{N+\tau} \rfloor
\end{equation}
Then for each $r\in\ZZ_{N+\tau}$ and each $l\in\ZZ_N$, we have
\begin{align*}
&\Theta\left(f^r\left(Y(C_{y_1\cdots y_n}^1, C_{w_1\cdots w_{n-1}}^2)\right), 
\phi, (n-3)N+l,2\delta\right)
\\\le&\sss(X,\tau,\delta)^{n-1}e^{\sum_{k=3}^{\theta_n}
\Phi\left(y_k,N\right)+2nN\eta+5N\phi^+}.
\end{align*}

\end{lemma}

\begin{proof}
Note that if $Q$ is an $((n-3)N+l,2\delta)$-separated set in 
$f^r(Y(C_{y_1\cdots y_n}^1, C_{w_1\cdots w_{n-1}}^2))$, then $f^{-r}(Q)$
contains an $((n-3)N+l+r,2\delta)$-separated set in
$Y(C_{y_1\cdots y_n}^1, C_{w_1\cdots w_{n-1}}^2)$. By Lemma \ref{uppnepest},
as 
$$(n-3)N+l+r\le (n-3)N+N+(N+\tau)\le nN,$$ 
we have
\begin{equation}\label{eq_frynepest}
\sss\left(f^r\left(Y(C_{y_1\cdots y_n}^1, C_{w_1\cdots w_{n-1}}^2)\right), 
(n-3)N+l,2\delta\right)\le\sss(X,\tau,\delta)^{n-1}.
\end{equation}

Suppose that $z\in f^r\left(Y(C_{y_1\cdots y_n}^1, C_{w_1\cdots w_{n-1}}^2)\right)$.
There is $y\in Y(C_{y_1\cdots y_n}^1, C_{w_1\cdots w_{n-1}}^2)$ such that
$f^r(y)=z$. 
By \eqref{eq_thetan}, we have
$$t_{\theta_n}+N\le \theta_n(N+\tau)+N\le(n-3)N+l$$
and
$$n-\theta_n<n-\left(\frac{(n-4)N}{N+\tau}-1\right)
<5+\frac{(n-4)\tau}{N+\tau}<5+\frac{n\tau}{N}.$$
Then by \eqref{eq_gonvar} and \eqref{eq_nnest5}, we have
\begin{align}
&\Phi\left(z,(n-3)N+l\right) \notag\\
\le&
\sum_{k=3}^{\theta_n}\Phi\left(f^{t_k}(y),N\right)
+((n-3)N+l-(\theta_n-2)N)\phi^+ \notag
\\\le&\sum_{k=3}^{\theta_n}\left(\Phi\left(y_k,N\right)
+N\eta\right)
+((n-3)N+N-(\theta_n-2)N)\phi^+ \notag
\\ 
\le&\sum_{k=3}^{\theta_n}\Phi\left(y_k,N\right)
+(\theta_n-2)N\eta
+(n-\theta_n)N\phi^+ \notag
\\\le&\sum_{k=3}^{\theta_n}\Phi\left(y_k,N\right)+nN\eta+
(5N\phi^++n\tau\phi^+) \notag
\\\label{eq_zpotest}
\le&\sum_{k=3}^{\theta_n}\Phi\left(y_k,N\right)+2nN\eta+5N\phi^+.
\end{align}
The lemma follows from \eqref{eq_frynepest} and \eqref{eq_zpotest}.
\end{proof}

\begin{lemma}\label{le_upppresest}
$P(\Lambda,\phi,2\delta)\le\al+\eta_0$.
\end{lemma}

\begin{proof}
By 
Lemma \ref{le_cylpresest} and \eqref{eq_epresest},
for $n>10$, we have
\begin{align*}
&\Theta(\Lambda,\phi,(n-3)N+l,2\delta)\\
\le&\sum_{r=0}^{N+\tau-1}\Theta(f^r(Y),\phi, (n-3)N+l,2\delta)
\\=&\sum_{r=0}^{N+\tau-1}\left(\sum_{E^n\times \ZZ_\tau^{n-1}}
\Theta\left(f^r\left(Y(C_{y_1\cdots y_n}^1, C_{w_1\cdots w_{n-1}}^2)\right), 
\phi, (n-3)N+l,2\delta\right)\right)
\\\le&(N+\tau)\tau^{n-1}\left(\sum_{(y_1,\cdots, y_n)\in E^n}
\sss(X,\tau,\delta)^{n-1}e^{\sum_{k=3}^{\theta_n}
\Phi\left(y_k,N\right)+2nN\eta+5N\phi^+}
\right)
\\\le&(N+\tau)\tau^{n-1}\sss(X,\tau,\delta)^{n-1}e^{2nN\eta+5N\phi^+}
\left(\sum_{x\in E}e^{\Phi(x,N)}
\right)^{\theta_n-2}
\\\le&(N+\tau)\tau^{n-1}\sss(X,\tau,\delta)^{n-1}e^{2nN\eta+5N\phi^++(n-2)N(\al+\eta)}.
\end{align*}

For every $m>10N$, there is $n>10$ and $l\in\ZZ_N$ such that
$m=(n-3)N+l$. 
So finally by \eqref{eq_epresest} and \eqref{eq_nnest5} we have
\begin{align*}
P(Y,\phi,2\delta)=&\limsup_{m\to\infty}
\frac{\ln\Theta(\Lambda,\phi,m,2\delta)}{m}
\\\le&\limsup_{n\to\infty}
\frac{\ln\Theta(\Lambda,\phi,(n-3)N+l,2\delta)}{(n-3)N}
\\\le&\frac{\ln\tau+\ln\sss(X,\tau,\delta)}{N}+2\eta+(\al+\eta)
\\<&\eta+2\eta+(\al+\eta)
\\<&\al+\eta_0.
\end{align*}

\end{proof}

\section{Proof of the Main Results}
\subsection{Without asymptotic entropy expansiveness}
\begin{proposition}\label{pr_ctdense}
Suppose that $(X,f,\phi)$ satisfies CT$(\delta,\gamma,\ep)$. Then
$$\overline{\spp_e(X,f,\phi)}\supset
\left[\min\left\{P^*(\phi),\pe(\phi,\ep)\right\}, P(\phi)\right].$$
\end{proposition}

\begin{proof}
Let 
$$(\al-\eta_0,\al+\eta_0)\subset
\left[\min\left\{P^*(\phi),\pe(\phi,\ep)\right\}, P(\phi)\right]$$
and $\eta=\frac{\eta_0}2>0$. By Proposition \ref{pr_lambda}, there is a compact
invariant set $\Lambda$ such that
$$|P(\Lambda,\phi,2\delta)-\al|<\eta.$$
Then $P(\Lambda,\phi,2\delta)>\al-\eta>\pe(\phi,\ep)$.
As $2\delta<\frac{\ep}{2}$, we can apply Lemma \ref{le_pest} to the subsystem
$(\Lambda,f)$ and obtain
$$P(\Lambda,\phi)=P(\Lambda,\phi,2\delta)\in(\al-\eta,\al+\eta).$$
Then by the Variational Principle, there is an ergodic measure $\mu$ supported
inside the compact invariant set $\Lambda$ such that
$$P_\mu(\phi)\in(\al-\eta_0,\al+\eta_0).$$
This implies that the proposition holds.
\end{proof}

The following is a consequence of Proposition \ref{pr_ctdense} and
Remark \ref{re_ctimct}.

\begin{corollary}\label{co_noaee}
Suppose that $(X,f,\phi)$ has the Climenhaga-Thompson structure. Then
$$\overline{\spp(\phi)}\supset
\left[\max\left\{P^*(\phi),\pe(\phi)\right\}, P(\phi)\right].$$
\end{corollary}


\subsection{Proof of Theorem \ref{thmain}}
Let 
$$(\al-\eta_0,\al+\eta_0)\subset
\left[\min\left\{P^*(\phi),\pe(\phi,\ep)\right\}, P(\phi)\right]$$
and $\eta=\frac{\eta_0}2>0$. By Remark \ref{re_ctimct} and asymptotically
entropy expansiveness, there are
$\delta,\gamma,\ep>0$ such that $(X,f,\phi)$ satisfies CT$(\delta,\gamma,\ep)$
and 
\begin{equation}\label{eq_tailbound}
h^*(f,2\delta)+\var(\phi,2\delta)<\eta.
\end{equation}
By Proposition \ref{pr_lambda}, there is a compact
invariant set $\Lambda$ such that
$$|P(\Lambda,\phi,2\delta)-\al|<\eta.$$
By \eqref{eq_tailbound} and Lemma \ref{le_eppres}, we have
$$P(\Lambda,\phi,2\delta)\le P(\Lambda,\phi)\le 
P(\Lambda,\phi,2\delta)+h^*(f,2\delta)+\var(\phi,2\delta)<\al+\eta_0.$$
By Lemma \ref{le_uppercts}, 
there is an ergodic measure $\mu$ supported
inside $\Lambda$ such that
$$P_\mu(\phi)=P(\Lambda,\phi)\in(\al-\eta_0,\al+\eta_0).$$
The result follows.

\section*{Acknowledgments}
The author is supported by National Natural Science Foundation of China (No. 11571387)
and CUFE Young Elite Teacher Project (No. QYP1902).
The author would like to thank Weisheng Wu for helpful discussions.



\end{document}